\documentclass[a4paper, 12pt, reqno]{amsart}
\usepackage[margin=1in]{geometry}
\usepackage{mathrsfs}
\usepackage{mathtools}
\usepackage{tikz, float, tikzscale}
\usepackage{pgfplots}
\pgfplotsset{compat=1.18}
\definecolor{mycolour}{rgb}{0.7,0.7,0.7}

\theoremstyle{definition}
\newtheorem*{defi}{Definition}
\theoremstyle{plain}
\newtheorem{thm}{Theorem}
\newtheorem*{thm*}{Theorem}
\newtheorem{prop}{Proposition}

\usepackage[T1]{fontenc}
\usepackage{libertinus,libertinust1math}

\title[Explicit Upper Bounds on Decay Rates of Fourier Transforms]{Explicit Upper Bounds on Decay Rates of Fourier Transforms of Self-similar Measures on Self-similar Sets}
\author{Ying Wai Lee}
\date{\today}

\begin{document}
\begin{abstract}
    The study of Fourier transforms of probability measures on fractal sets plays an important role in recent research. Faster decay rates are known to yield enhanced results in areas such as metric number theory. This paper focuses on self-similar probability measures defined on self-similar sets. Explicit upper bounds are derived for their decay rates, improving upon prior research. These findings are illustrated with an application to sets of numbers whose digits in their L\"uroth representations are restricted to a finite set.
\end{abstract}
\maketitle
\tableofcontents
\section{Introduction \& Results}
The Fourier transform, a fundamental tool in harmonic analysis, is initially defined for functions to analyse their frequency components. It is later generalised to measures, enabling the frequency analysis of distributions. Let $F\subset[0,1]$ and $\mu$ be a probability measure on $F$. Define $\widehat{\mu}:\mathbb{R}\to\mathbb{C}$, the Fourier transform of $\mu$, by, for any $\xi\in\mathbb{R}$,
\begin{align*}
    \widehat{\mu} (\xi)\coloneqq\int_{F}e^{-2\pi i\xi x}\,\mu(\mathrm{d}x).
\end{align*}
One research focus is the study of probability measures with the Rajchman property; that is, measures whose Fourier transform vanishes at infinity:
\begin{align}
    \label{eq: Rajchman}
    \lim_{\xi\to\pm\infty}\widehat{\mu}(\xi)=0.
\end{align}
The Rajchman property of measures reflects a form of regularity at infinity; the measure does not exhibit excessive concentration at any specific scale.

Although one can be interested in establishing the exact convergence rate in \eqref{eq: Rajchman}, knowing some upper bounds of the rate can provide certain information on some metric properties of the set. Let $\Phi:\mathbb{N}\to\mathbb{R}^+$ be a function. $\Phi$ is said to be a decay rate of $\mu$ if, as $\xi\to\pm\infty$,
\begin{align}
    \label{eq: Lyons assumption 0}
    \widehat{\mu}(\xi)=O\left(\Phi\left(|\xi|\right)\right).
\end{align}
In the case of $\lim_{n\to+\infty}\Phi(n)=0$, \eqref{eq: Lyons assumption 0} quantifies the convergence rate in \eqref{eq: Rajchman}.

The classical results of Lyons \cite[Theorems 3 \& 4]{LyonsRussell1986Tmon} provide foundational insights into how the decay rate of a probability measure is useful in metric number theory. These two results can also be compared to highlight how faster decay can lead to stronger conclusions.
\begin{thm*}[Lyons, 1986; {\cite[Theorem 4]{LyonsRussell1986Tmon}}]
    Let $F\subset[0,1]$ and $\mu$ be a probability measure on $F$. Suppose, there exists a non-increasing $\Phi:\mathbb{N}\to\mathbb{R}^+$ such that as $\xi\to\pm\infty$, \eqref{eq: Lyons assumption 0} holds and
    \begin{align}
        \label{eq: Lyons assumption 4}
        \sum_{n=2}^\infty\frac{\Phi(n)}{n\log{n}}<+\infty.
    \end{align}
    Then, for $\mu$-almost every $x\in F$ and any sequence $(q_n)_{n\in\mathbb{N}}$ of positive integers, if $\inf_{n\in\mathbb{N}}{q_{n+1}}/{q_n}>1$, then the sequence $(q_nx\mod1)_{n\in\mathbb{N}}$ is uniformly distributed in $[0,1)$.
\end{thm*}
In the case of $\Phi$ exhibiting logarithmic decay (i.e., there exists $\varepsilon>0$ such that as $n\to+\infty$, $\Phi(n)=O(\log^{-\varepsilon}{n})$), condition~\eqref{eq: Lyons assumption 4} is satisfied. For $\Phi$ with faster decay, such as power-law decay (i.e., there exists $\varepsilon>0$ such that as $n\to+\infty$, $\Phi(n)=O(n^{-\varepsilon})$), the conclusion of the above result can be strengthened, allowing for broader choices of sequences.
\begin{thm*}[Lyons, 1986; {\cite[Theorem 3]{LyonsRussell1986Tmon}}]
    Let $F\subset[0,1]$ and $\mu$ be a probability measure on $F$. Suppose, there exists a non-increasing $\Phi:\mathbb{N}\to\mathbb{R}^+$ such that as $\xi\to\pm\infty$, \eqref{eq: Lyons assumption 0} holds and
    \begin{align}
        \label{eq: Lyons assumption 3}
        \sum_{n=1}^\infty\frac{\Phi(n)}{n}<+\infty.
    \end{align}
    Then, for $\mu$-almost every $x\in F$ and any strictly increasing sequence $(q_n)_{n\in\mathbb{N}}$ of positive integers, the sequence $(q_nx\mod1)_{n\in\mathbb{N}}$ is uniformly distributed in $[0,1)$.
\end{thm*}
These results illustrate how faster decay has the potential to yield stronger conclusions. Therefore, deriving faster decay rates stands as an important research focus.

In the works of Kaufman \cite{kaufman1980continued} and Queff{\'e}lec--Ramar{\'e} \cite{queffelec2003analyse}, sets of numbers whose partial quotients in their continued fraction representations are restricted to a finite set $\mathcal{A}\subset\mathbb{N}$ have been studied. These work prove that each of these sets (provided their Hausdorff dimensions exceed certain thresholds) supports a probability measure $\mu$ satisfying both conditions~\eqref{eq: Lyons assumption 0} and \eqref{eq: Lyons assumption 3} for some $\Phi$ with power-law decay.

For any $x\in[0,1)\setminus\mathbb{Q}$, there exists a unique sequence $(a_n)_{n\in\mathbb{N}}$ of positive integers, referred to as partial quotients, such that
\begin{align}
    x
    &=[a_1,a_2,a_3,\ldots] \label{eq: CF Repr}\\
    &\coloneqq{\frac{1}{\displaystyle a_1+\frac{1}{\displaystyle a_2+\frac{1}{a_3+\cdots}}}}; \label{eq: CF Expa}
\end{align}
where equation~\eqref{eq: CF Repr} is referred to as the continued fraction representation of $x$, and equation~\eqref{eq: CF Expa} as the continued fraction expansion of $x$. Define, for any $\mathcal{A}\subset\mathbb{N}$, $F_\mathcal{A}\coloneqq\{[a_1,a_2,a_3,\ldots]:\text{$a_n\in\mathcal{A}$ for all $n\in\mathbb{N}$}\}$; and for any $N\in\mathbb{N}$, $F_N\coloneqq F_{\mathbb{N}\cap[1,N]}$. These sets play an important role in the study of Diophantine approximation. According to \cite[Theorem 1.4]{beresnevich2016metricdiophantineapproximationaspects}, the countable union $\bigcup_{N\in\mathbb{N}}F_N$ coincides with the set of all badly approximable numbers in $[0,1]$. 

\begin{thm*}[Kaufman, 1980; \cite{kaufman1980continued}]
    Let $N\in\mathbb{N}\setminus\{1\}$. Suppose, the Hausdorff dimension of $F_N$, $\dim{F_N}>2/3$. Then, there exists a probability measure $\mu_N$ on $F_N$ such that as $\xi\to\pm\infty$,
    \begin{align*}
        \widehat{\mu_{{N}}}(\xi)=O\left(|\xi|^{-0.0007}\right).
    \end{align*}
\end{thm*}
\begin{thm*}[Queff\'elec--Ramar\'e, 2003; {\cite[Th\'eor\`eme 1.4]{queffelec2003analyse}}]
    Let $\mathcal{A}\subset\mathbb{N}$ be
    finite. Suppose, $\dim{F_{\mathcal{A}}}>1/2$. Then, for any $\varepsilon>0$ and $1/2<\delta<\dim{F_{\mathcal{A}}}$, there exists a probability measure $\mu_{\mathcal{A},\varepsilon,\delta}$ on $F_{\mathcal{A}}$ such that as $\xi\to\pm\infty$,
    \begin{align*}
        \widehat{\mu_{\mathcal{A},\varepsilon,\delta}}(\xi)
        =O\left(|\xi|^{-\eta+\varepsilon}\right),
    \end{align*}
    where $\eta\coloneqq{\delta(2\delta-1)}/{((2\delta+1)(4-\delta))}>0$.
\end{thm*}
The result of Queff\'elec--Ramar\'e improves upon that of Kaufman by generalising from \( F_N \) to \( F_{\mathcal{A}} \), providing faster decay rates with explicit parametrisation, and relying on a weaker threshold for the Hausdorff dimension. By the approximation of Hensley, \cite[Table 1]{hensley1996polynomial} or \cite[(4.10)]{HENSLEY1989182},
\begin{align*}
    1/2<\dim{F_2}=0.5312805\ldots<2/3;
\end{align*}
the result of Queff\'elec--Ramar\'e applies to more cases than that of Kaufman. Both results construct probability measures satisfying both conditions~\eqref{eq: Lyons assumption 0} and \eqref{eq: Lyons assumption 3}, where $\Phi$ exhibits power-law decay. The result of Lyons \cite[Theorem 3]{LyonsRussell1986Tmon} applies for these sets.

In the work of Li--Sahlsten \cite{li2022trigonometric}, self-similar probability measures on self-similar sets are studied. One of their results \cite[Theorem 1.3]{li2022trigonometric} focuses on non-singleton self-similar sets satisfying a non-Liouville condition. This result states that every self-similar probability measure $\mu$ on such sets satisfies both condition~\eqref{eq: Lyons assumption 0} and \eqref{eq: Lyons assumption 4} for some $\Phi$ with logarithmic decay.

\begin{defi}[Self-similar Set]
    Let $F\subset[0,1]$. $F$ is said to be a self-similar set if, there exist a finite set $\mathcal{A}$ and its collection of similitudes $f_w:[0,1]\to[0,1]$ such that
    \begin{align}
    \label{eq: SSS}
        F=\bigcup_{w\in\mathcal{A}}f_w(F),
    \end{align}
    where for any $w\in\mathcal{A}$, there exist a contraction ratio $r_w\in(0,1)$ and a translation $b_w\in[0,1]$ such that for any $x\in[0,1]$,
    \begin{align}
    \label{eq: IPS}
        f_w(x)
        = r_wx+b_w.
    \end{align}
\end{defi}
\begin{defi}[Self-similar Measure]
    Let $\mu$ be a probability measure on a subset of $[0,1]$. $\mu$ is said to be self-similar if, there exist a finite set $\mathcal{A}$ and its collection of similitudes \eqref{eq: IPS} such that
    \begin{align}
    \label{eq: SSM}
        \mu=\sum_{w\in\mathcal{A}}p_w\mu\circ {f_w}^{-1},
    \end{align}
    where the weights $(p_w)_{w\in\mathcal{A}}$ satisfy $\sum_{w\in\mathcal{A}}p_w=1$ and for any $w\in\mathcal{A}$, $p_w>0$.
\end{defi}
According to \cite[Theorem 3.1.(3)(i)]{hutchinson1981fractals}, for any finite collection of similitudes, there is a unique non-empty compact subset of $[0,1]$ satisfying the self-similarity property~\eqref{eq: SSS}. According to the Hutchinson Theorem \cite[Theorem 4.4.(1)(i)]{hutchinson1981fractals}, for any finite collection of similitudes and weights, there is a unique probability measure satisfying the self-similarity property~\eqref{eq: SSM}.

In Diophantine approximation, the concept of non-Liouville numbers (sometimes called diophantine numbers) is a natural generalisation of badly approximable numbers.
\begin{defi}[Non-Liouville Number]
Let $\theta\in\mathbb{R}$ and $l>0$. $\theta$ is said to be non-Liouville of degree $l$ if, there exists $c>0$ such that for any $p\in\mathbb{Z}$ and $q\in\mathbb{N}$,
\begin{align}\label{eq: non-Liouville}
    \left|\theta-\frac{p}{q}\right|\geq\frac{c}{q^l}.
\end{align}
\end{defi}
Badly approximable numbers are precisely the non-Liouville numbers of degree 2. By the Dirichlet Theorem \cite[Theorem 1.2]{beresnevich2016metricdiophantineapproximationaspects}, every non-Liouville number is of degree at least 2.

\begin{thm*}[Li--Sahlsten, 2022; {\cite[Theorem 1.3]{li2022trigonometric}}]
    Let $F\subset[0,1]$ be a self-similar set of the form \eqref{eq: SSS} associated with similitudes \eqref{eq: IPS} for a finite set $\mathcal{A}$.
    Suppose, $F$ is not a singleton, and there exist $j,k\in\mathcal{A}$ and $l>2$ such that $\log{r_j}/\log{r_k}$ is non-Liouville of degree $l$. Then, for any self-similar measure $\mu$ on $F$, there exists $\beta>0$ such that as $\xi\to\pm\infty$,
    \begin{align}
    \label{eq: Li-S Theorem 1.3}
        \widehat{\mu}(\xi)
        =O\left(\log^{-\beta}{|\xi|}\right).
    \end{align}
\end{thm*}

Although \eqref{eq: Li-S Theorem 1.3} is sufficient for both conditions~\eqref{eq: Lyons assumption 0} and \eqref{eq: Lyons assumption 4}, the parameter $\beta$ is implicit, and its explicit dependence on the probability measure and the self-similar set remains undetermined. Determining an explicit value for $\beta$ is crucial for certain applications in the metric theory of Diophantine approximation. For instance, to ensure sufficient decay, the result of Pollington--Velani--Zafeiropoulos--Zorin (PVZZ) \cite[Theorem 1]{pollington2020inhomogeneous} requires $\beta>2$ as a a threshold. However, the implicit nature of $\beta$ in \eqref{eq: Li-S Theorem 1.3} prevents verification of whether the threshold can be met, limiting its applicability.

This paper establishes an explicit expression of $\beta$ in \eqref{eq: Li-S Theorem 1.3} for general self-similar probability measure $\mu$. The formula of $\beta$ is provided in terms of the upper regularity exponents of $\mu$ and the non-Liouville degree $l$ of the logarithmic ratio of the contraction ratios. The value of $\beta$ is maximised for self-similar sets satisfying the open set condition (a pairwise disjoint condition). An application is demonstrated for sets of numbers whose digits in their L\"uroth representations are restricted to a finite set, analogous to continued fraction representations.

\begin{defi}[Upper Regularity Exponent]
    Let $\mu$ be a probability measure on a subset of $[0,1]$ and $\alpha\in[0,1]$. $\alpha$ is said to be an upper regularity exponent of $\mu$ if, there exist $C_{\alpha}>0$ and $r_{\alpha}>0$ such that for any interval $I\subset [0,1]$, if $\operatorname{diam}{I}<r_\alpha$ then
    \begin{align}
        \label{eq: URE}
        \mu(I)\leq C_{\alpha}\left(\operatorname{diam}{I}\right)^\alpha,
    \end{align}
    where $\operatorname{diam}{I}$ denotes the diameter of $I$.
\end{defi}

Theorems~\ref{thm: 1}, \ref{thm: 2}, \ref{thm: 3}, and \ref{thm: 4} are the main results of this paper. Theorem~\ref{thm: 1} provides an explicit decay rate to refine and improve the previous result \cite[Theorem 1.3]{li2022trigonometric}. Theorem~\ref{thm: 2} provides a method for computing the Hausdorff dimension of self-similar sets satisfying the open set condition. Theorem~\ref{thm: 3} provides a probability measure achieving the fastest decay rate under Theorem~\ref{thm: 1}. Theorem~\ref{thm: 4} is an application of Theorem~\ref{thm: 3} to number theory.

\begin{thm}\label{thm: 1}
    Let $F\subset[0,1]$ be a self-similar set of the form \eqref{eq: SSS} associated with similitudes \eqref{eq: IPS} for a finite set $\mathcal{A}$.
    Suppose, there exist $j,k\in\mathcal{A}$ and $l\geq 2$ such that $\log{r_j}/\log{r_k}$ is non-Liouville of degree $l$. Then, for any self-similar probability measure $\mu$ on $F$ and any  upper regularity exponent $\alpha\in[0,1]$ of $\mu$, as $\xi\to\pm\infty$,
    \begin{align*}
        \widehat{\mu}(\xi)
        =O\left(\log^{-\alpha/(2(1+2\alpha)(8l-7))}{|\xi|}\right)
        .
    \end{align*}
\end{thm}

According to the Frostman Lemma \cite[Theorem 8.8]{MattilaPertti1995Gosa}, for any probability measure $\mu$ on $F$, every upper regularity exponent $\alpha$ of $\mu$ is at most $\dim{F}$, the Hausdorff dimension of $F$. In the case of $F$ being a singleton, $\dim{F}=0$ implies that the only upper regularity exponent is $\alpha=0$. Theorem~\ref{thm: 1} remains consistent in degenerate cases by providing the trivial non-decay rate, after the non-singleton assumption is removed from the result of Li--Sahlsten {\cite[Theorem 1.3]{li2022trigonometric}}.

In the case of including the non-singleton assumption, the result of Feng--Lau \cite[Proposition 2.2]{FENG2009407} ensures that for any self-similar probability measure, a positive upper regularity exponent exists. Theorem~\ref{thm: 1} guarantees the decay of its Fourier transform in non-degenerate cases.

In the case of including the open set condition (a pairwise disjoint condition), a self-similar probability measure $\mu_{\mathcal{A}}$ can be constructed, so that the maximal upper regularity exponent equals $\dim{F}$. This measure maximises the decay rate described in Theorem~\ref{thm: 1}.

\begin{thm}\label{thm: 2}
    Let $F\subset[0,1]$ be a self-similar set of the form \eqref{eq: SSS} associated with similitudes \eqref{eq: IPS} for a non-empty finite set $\mathcal{A}$.
    Suppose, the images $(f_w([0,1]))_{w\in\mathcal{A}}$ are pairwise disjoint except at the endpoints.
    Then, $\dim{F}$ is the unique solution $s=s_{\mathcal{A}}$ in $[0,1]$ to the following equation:
    \begin{align}
    \label{eq: sum rw s =1}
        \sum_{w\in\mathcal{A}}{r_w}^{s}=1.
    \end{align}
\end{thm}

\begin{thm}\label{thm: 3}
    Let $F\subset[0,1]$ be a self-similar set of the form \eqref{eq: SSS} associated with similitudes \eqref{eq: IPS} for a finite set $\mathcal{A}$.
    Suppose, the images $(f_w([0,1]))_{w\in\mathcal{A}}$ are pairwise disjoint except at the endpoints, and there exist $j,k\in\mathcal{A}$ and $l\geq 2$ such that $\log{r_j}/\log{r_k}$ is non-Liouville of degree $l$.
    Then, there exists a self-similar probability measure $\mu_{\mathcal{A}}$ on $F$ such that as $\xi\to\pm\infty$,
    \begin{align*}
        \widehat{\mu_{\mathcal{A}}}(\xi)
        =O\left(\log^{-\dim{F}/(2(1+2\dim{F})(8l-7))}{|\xi|}\right)
        .
    \end{align*}
\end{thm}

Theorem~\ref{thm: 2} is essentially the same as the classical result of Moran \cite[Theorem II]{Moran_1946}. Different from the original approach of Moran, which constructs additive functions and applies density lemmas, the proof of Theorem~\ref{thm: 2} employs a novel approach. It first constructs a series of self-similar probability measures and then selects the one maximising its upper regularity exponent. This measure is sufficient to deduce Theorem~\ref{thm: 3} from Theorem~\ref{thm: 1}.

Theorems~\ref{thm: 2} and \ref{thm: 3} can be applied together. By equation~\eqref{eq: sum rw s =1}, the Hausdorff dimension of the self-similar set can be computed or approximated numerically. Once the non-Liouville degree is known, the decay rate described in Theorem \ref{thm: 3} can be calculated. 

Theorem~\ref{thm: 3} provides a framework for number theorists studying sets of numbers with self-similar structures. One application lies in the study of sets of numbers whose digits in their L\"uroth representation are restricted to a finite set. This result is an analogous to the work of Queff\'elec--Ramar\'e \cite{queffelec2003analyse} on restricted partial quotients.

For any $x\in(0,1]$, there exists a unique sequence $(d_n)_{n\in\mathbb{N}}$ of positive integers not equal to 1, referred to as digits, such that
\begin{align}
    x
    &=[d_1,d_2,\ldots,d_n,\ldots] \label{eq: Luroth Repr}\\
    &\coloneqq\sum_{k=1}^{\infty}\frac{1}{d_k\prod_{j=1}^{k-1}d_j(d_j-1)}; \label{eq: Luroth Expa}
\end{align}
where equation~\eqref{eq: Luroth Repr} is referred to as the L\"uroth representation of $x$, and equation~\eqref{eq: Luroth Expa} as the L\"uroth expansion of $x$. The notation $[d_1,d_2,\ldots,d_n,\ldots]$ denotes the L\"uroth representation when the brackets contain $d$'s instead of $a$'s, distinguishing it from the continued fraction representation $[a_1, a_2, \ldots,a_n,\ldots]$, 

Define, for any $\mathcal{A}\subset\mathbb{N}\setminus\{1\}$, $L_{\mathcal{A}}$ to be the sets of numbers whose digits in their L\"uroth representations are restricted to $\mathcal{A}$. Formally,
\begin{align*}
    L_{\mathcal{A}}
    \coloneqq\left\{[d_1,d_2,\ldots,d_n,\ldots]:\text{$d_n\in\mathcal{A}$ for all $n\in\mathbb{N}$}\right\}.
\end{align*}
For any finite $\mathcal{A}\subset\mathbb{N}\setminus\{1\}$, $L_{\mathcal{A}}$ is self-similar in $(0,1]$. Specifically, $L_{\mathcal{A}}$ can be expressed as
\begin{align}
\label{eq: Luroth SSS}
    L_{\mathcal{A}}=\bigcup_{d\in\mathcal{A}}f_d\left(L_{\mathcal{A}}\right),
\end{align}
where for any $d\in\mathbb{N}\setminus\{1\}$, the similitude $f_d:(0,1]\to(0,1]$ is defined by, for any $x\in(0,1]$,
\begin{align}
\label{eq: Luroth IFS}
    f_d(x)\coloneqq\frac{1}{d}+\frac{1}{d(d-1)}x,
\end{align}
derived from the following recursion formula: for any sequence $(d_n)_{n\in\mathbb{N}}$ in $\mathbb{N}\setminus\{1\}$,
\begin{align*}
    [d_1,d_2,d_3,\ldots]=\frac{1}{d_1}+\frac{1}{d_1(d_1-1)}[d_2,d_3,d_4,\ldots].
\end{align*}

Figure~\ref{fig: L3} illustrates the fractal construction process of $L_3\coloneqq L_{\{2,3\}}$. Starting from the interval $(0,1]$, sub-intervals are partitioned and removed iteratively at each level. At the first level, $(0,1]$ is partitioned to three sub-intervals: $(0,1/3]$, $(1/3,1/2]$, and $(1/2,1]$, which correspond to numbers in $(0,1]$ with $d_1\geq4$, $d_1=3$, and $d_1=2$ respectively. The first sub-interval is removed, while the second and third are retained for the next level. At each subsequent level $n\in\mathbb{N}$, every retained interval is further partitioned to three sub-intervals, with diameters in the ratio of 2$:$1$:$3, corresponding to $d_n\geq4$, $d_n=3$, and $d_n=2$ respectively. As before, the first sub-interval is removed, and the remaining two are retained to continue the process. The set $L_3$ is the countable intersection of all levels in this iterative construction.

\input{figure_L3_1127}

The non-Liouville condition in Theorem~\ref{thm: 3} can be proved by applying the explicit Baker Theorem of Matveev \cite[Theorem 2.1]{matveev2000explicit} for any non-singleton $\mathcal{A}\subset\mathbb{N}\setminus\{1\}$. Theorem~\ref{thm: 4} is established as an application of Theorem~\ref{thm: 3} to number theory. 

\begin{thm}\label{thm: 4}
    For any non-singleton finite $\mathcal{A}\subset\mathbb{N}\setminus\{1\}$, there exists a self-similar probability measure $\mu_{\mathcal{A}}$ on $L_{\mathcal{A}}$ such that as $\xi\to\pm\infty$,
    \begin{align*}
        \widehat{\mu_{\mathcal{A}}}(\xi)=O\left(\log^{-\beta_{\mathcal{A}}}{|\xi|}\right),
    \end{align*}
    where $a_1$ and $a_2$ are the first smallest and second smallest elements in $\mathcal{A}$ respectively, and
    \begin{align*}
        \beta_{\mathcal{A}}
        \coloneqq
        \frac{\dim{L_\mathcal{A}}}{1+2\dim{L_\mathcal{A}}}\frac{10^{-10}}{\log{a_1}\log{a_2}+1}
        >0
        .
    \end{align*}
\end{thm}
By Theorem~\ref{thm: 2}, the Hausdorff dimension of $L_3$ is approximated as
\begin{align*}
    \dim{L_3}
    =0.60096685161367548572
    \ldots
    .
\end{align*}
By Theorem~\ref{thm: 4}, there exists a probability measure $\mu_3$ on $L_3$ such that as $\xi\to\pm\infty$,
\begin{align*}
    \widehat{\mu_3}(\xi)=O\left(\log^{-\beta_3}{|\xi|}\right),
\end{align*}
where $\beta_3>10^{-11}>0$.

Returning to the motivation of establishing an explicit expression for $\beta$ in \eqref{eq: Li-S Theorem 1.3}, the result of PVZZ \cite[Theorem 1]{pollington2020inhomogeneous} sets a significant target by requiring $\beta>2$ as a threshold in its assumption. Under the most favourable conditions, Theorem~\ref{thm: 3}, which achieves the fastest decay rate derived from Theorem~\ref{thm: 1}, yields $\beta=1/54$. This value reflects inherent constraints, such as the Hausdorff dimension of subsets in $[0,1]$ being at most 1 and the degree of non-Liouville numbers being at least 2. While the results in this paper do not yet meet this threshold, they offer an explicit framework and improved decay rates that contribute to bridging the gap toward the target. 

\section{Proof of Theorem~\ref{thm: 1}}

The proof of Theorem~\ref{thm: 1} employs quantitative renewal theory, specifically for the stopping time of random walks, as introduced in prior work \cite{li2022trigonometric}. A few definitions and propositions are presented before starting the main proof of Theorem~\ref{thm: 1}. 

The self-similar measure $\mu$, defined as in equation~\eqref{eq: SSM} associated with similitudes \eqref{eq: IPS}, can be expressed in terms of compositions of similitudes and their corresponding weights. Define, for any $n\in\mathbb{N}$ and $w=(w_1,w_2,\ldots,w_n)\in\mathcal{A}^n$, a similitude $f_w:[0,1]\to[0,1]$ by 
\begin{align}
    \label{eq: f_w}
    f_w\coloneqq f_{w_n}\circ \cdots \circ f_{w_2} \circ f_{w_1},
\end{align}
along with the contraction ratio $r_w\coloneqq r_{w_1} r_{w_2} \cdots r_{w_n}$, and the weight $p_w\coloneqq p_{w_1} p_{w_2} \cdots p_{w_n}$.
By these definitions, for any $n\in\mathbb{N}$,
\begin{align*}
    \mu=\sum_{w\in\mathcal{A}^n}p_w\mu\circ {f_w}^{-1}.
\end{align*}
Define, for any $t>0$, the stopping time $n_t:\mathcal{A}^\mathbb{N}\to\mathbb{N}$ by, for any $w=(w_n)_{n\in\mathbb{N}}\in\mathcal{A}^{\mathbb{N}}$,
\begin{align*}
    n_t(w)
    \coloneqq\min\left\{n\in\mathbb{N}:-\log{r_w}\geq t\right\}.
\end{align*}
Define, for any $t>0$,
\begin{align*}
    \mathcal{W}_t
    \coloneqq\left\{\left(w_1,w_2,\ldots w_{n_t(w)}\right):(w_n)_{n\in\mathbb{N}}\in\mathcal{A}^\mathbb{N}\right\}.
\end{align*}

Proposition~\ref{prop: refined Li-S Lemma 3.1}, as the same as {\cite[Lemma 3.1]{li2022trigonometric}}, serves as a starting point of the proof.
\begin{prop}[{\cite[Lemma 3.1]{li2022trigonometric}}]\label{prop: refined Li-S Lemma 3.1}
    For any $\xi\in\mathbb{R}$ and $t>0$,
    \begin{align*}
        \left|\widehat{\mu}(\xi)\right|^2
        \leq
        \iint_{[0,1]^2}\sum_{w\in \mathcal{W}_t}p_we^{-2\pi i \xi(f_w(x)-f_w(y))}\,\mu(\mathrm{d}x)\,\mu(\mathrm{d}y)
        .
    \end{align*}
\end{prop}
The integral in Proposition~\ref{prop: refined Li-S Lemma 3.1} is partitioned into two regions, depending on whether the points are near the diagonal or not. Upper bounds for the integrals over each region are established, which together suffice to prove Theorem~\ref{thm: 1}.

Proposition~\ref{prop: refined Li-S Proposition 3.2} is an applied version of {\cite[Proposition 3.2]{li2022trigonometric}}. The integral over points near the diagonal is controlled by the upper regularity exponent of the measure.
\begin{prop}\label{prop: refined Li-S Proposition 3.2}
    Let $\alpha\in[0,1]$ be an upper regularity exponent of $\mu$ and $l\geq2$. There exist $D_\alpha>0$ and $t_\alpha>0$ such that for any $\xi\in\mathbb{R}$ and $t>t_\alpha$,
    \begin{align*}
        \left| \iint_{A_{t}} \sum_{w \in \mathcal{W}_t} p_w e^{-2 \pi i \xi (f_w(x) - f_w(y))} \,\mu(\mathrm{d}x)\,\mu(\mathrm{d}y)\right|
        \leq D_\alpha t^{-\alpha/((1+2\alpha)(8l-7))},
    \end{align*}
    where $A_{t}$ is the set of points near the diagonal, defined as
    \begin{align}
        \label{eq: def of A}
         A_{t}
         \coloneqq
         \left\{(x,y)\in[0,1]^2:|x-y|\leq t^{-1/((1+2\alpha)(8l-7))}\right\}.
    \end{align}
\end{prop}
\begin{proof}    
    Since $\alpha\in[0,1]$ is an upper regularity exponent of $\mu$, there exist $C_{\alpha}>0$ and $r_\alpha>0$ such that for any interval $I\subset[0,1]$, if $\operatorname{diam}{I}<r_\alpha$ then inequality~\eqref{eq: URE} holds. Define
    \begin{align*}
        t_\alpha\coloneqq {r_\alpha}^{-(1+2\alpha)(8l-7)}>0.
    \end{align*}
    Thus, for any $t>t_\alpha$, it satisfies $0<\delta\coloneqq t^{-1/((1+2\alpha)(8l-7))}<r_\alpha$.
    
    Note that for any $\theta\in\mathbb{R}$, $|e^{i\theta}|=1$; and for any $t>0$, $\sum_{w\in\mathcal{W}_t}p_w=1$. By applying the triangle inequality for integrals, the Fubini Theorem, inequality~\eqref{eq: URE}, and the definition of $\delta$,
    \begin{align*}
        \left| \iint_{A_{t}} \sum_{w \in \mathcal{W}_t} p_w e^{-2 \pi i \xi (f_w(x) - f_w(y))} \,\mu(\mathrm{d}x)\,\mu(\mathrm{d}y)\right|
        &\leq \iint_{A_{t}} \sum_{w \in \mathcal{W}_t} p_w \left|e^{-2 \pi i \xi (f_w(x) - f_w(y))}\right| \,\mu(\mathrm{d}x)\,\mu(\mathrm{d}y) \\
        &=\iint_{A_{t}}1\,\mu(\mathrm{d}x)\,\mu(\mathrm{d}y)
        =\int_{[0,1]}\mu((x-\delta,x+\delta))\,\mu(\mathrm{d}x) \\
        &\leq C_{\alpha}\left(2\delta\right)^{\alpha}
        =D_{\alpha}t^{-\alpha/((1+2\alpha)(8l-7))}
        ,
    \end{align*}
    where $D_{\alpha}\coloneqq 2^{-\alpha/((1+2\alpha)(8l-7))}C_{\alpha}>0$.
\end{proof}

Quantitative renewal theory is applied to control the integral over points that are not near the diagonal. The non-Liouville condition is employed to guarantee that an auxiliary measure is weakly diophantine.

\begin{defi}[Weakly Diophantine]
Let $\lambda$ be a probability measure on $\mathbb{R}^+$ with finite support $\Lambda$ and $l>0$. $\lambda$ is said to be $l$-weakly diophantine if,
\begin{align*}
    \liminf_{b\to\pm\infty}|b|^l\left|1-\mathscr{L}\lambda(ib)\right|>0,
\end{align*}
where $\mathscr{L}\lambda$ is the Laplace transform of $\lambda$; that is, for any $z\in\mathbb{C}$,
\begin{align*}
    \mathscr{L}\lambda(z)\coloneqq\int_{\Lambda}e^{-zx}\,\lambda(\mathrm{d}x).
\end{align*}
\end{defi}

In the case of a self-similar measure $\mu$ of the form \eqref{eq: SSM} associated with similitudes \eqref{eq: IPS}, define the corresponding auxiliary measure $\lambda$ by
\begin{align}
    \label{eq: lambda}
    \lambda\coloneqq\sum_{w\in\mathcal{A}}p_w\delta_{-\log{r_w}},
\end{align}
where $\delta$ is the Dirac delta measure. Since $\mathcal{A}$ is finite, $\Lambda=\{-\log{r_w}>0:w\in\mathcal{A}\}$, the support of $\lambda$, is also finite. Proposition~\ref{prop: refined Li-S Lemma 3.4} refines \cite[Lemma 3.4]{li2022trigonometric}.
\begin{prop}\label{prop: refined Li-S Lemma 3.4}
    Suppose, there exist $j,k\in\mathcal{A}$ and $l\geq 2$ such that $\log{r_j}/\log{r_k}$ is non-Liouville of degree $l$. Then, the probability measure $\lambda$, defined in \eqref{eq: lambda}, is $(2l-2)$-weakly diophantine.
\end{prop}
\begin{proof}
Pick any $b_1\in\mathbb{Z}\setminus\{0\}$. By taking $b\coloneqq 2\pi b_1/\log{r_k}$,
\begin{align*}
    \left|1-\mathscr{L}\lambda(ib)\right|
    &\geq\left|\Re{\left(p_j\left(1-e^{-b\log{r_j}}\right)+p_k\left(1-e^{-b\log{r_k}}\right)\right)}\right| \\
    &\gg \max{\left(\operatorname{dist}\left(b\log{r_j},2\pi\mathbb{Z}\right)^2,\operatorname{dist}\left(b\log{r_k},2\pi\mathbb{Z}\right)^2\right)} \\
    &\gg\max{\left(\operatorname{dist}\left(\frac{\log{r_j}}{\log{r_k}}b_1,\mathbb{Z}\right)^2,\operatorname{dist}\left(b_1,\mathbb{Z}\right)^2\right)},
\end{align*}
where the Vinogradov symbol $\gg$ denotes a quantity being greater than a fixed positive multiple of another.

By $\theta\coloneqq\log{r_j}/\log{r_k}$ being non-Liouville of degree $l$, inequality~\eqref{eq: non-Liouville} implies that there exists $c>0$ such that for any $p\in\mathbb{Z}$ and $q\in\mathbb{N}$, 
\begin{align*}
    \left|\frac{\log{r_j}}{\log{r_k}}-\frac{p}{q}\right|\geq\frac{c}{q^l}.
\end{align*}
In particular, by taking $q=|b_1|\in\mathbb{N}$ and multiplying both sides by $|b_1|>0$, for any $p\in\mathbb{Z}$,
\begin{align*}
    \left|\frac{\log{r_j}}{\log{r_k}} b_1-\frac{pb_1}{|b_1|}\right|
    \geq\frac{c|b_1|}{{|b_1|}^l}
    =\frac{c}{{|b_1|}^{l-1}}.
\end{align*}
Thus, for any $p\in\mathbb{Z}$,
\begin{align*}
    \left|\frac{\log{r_j}}{\log{r_k}} b_1-p\right|&\geq\frac{c}{{|b_1|}^{l-1}};
\end{align*}
it follows that,
\begin{align*}
    \left|1-\mathscr{L}\lambda(ib)\right|
    \gg
    \operatorname{dist}\left(\frac{\log{r_j}}{\log{r_k}}b_1,\mathbb{Z}\right)^2
    \gg |b_1|^{2-2l}
    \gg |b|^{2-2l}
    .
\end{align*}

Thus, for any $b_1\in\mathbb{Z}\setminus\{0\}$ and $b=2\pi b_1/\log{r_j}\ne0$, the above implies that 
\begin{align*}
    |b|^{2l-2}\left|1-\mathscr{L}\lambda(ib)\right|\gg1.
\end{align*}
Therefore, $\lambda$ is $(2l-2)$-weakly diophantine.
\end{proof}

Let $\sigma$ be the expectation of the auxiliary measure $\lambda$, defined in \eqref{eq: lambda}; that is,
\begin{align*}
    \sigma\coloneqq\int_{\Lambda}x\,\lambda(\mathrm{d}x)>0,
\end{align*}
where $\Lambda$ is the support of $\lambda$. Let $(X_n)_{n\in\mathbb{N}}$ be a sequence of independent and identically distributed random variables in $\mathbb{R}$ with probability distribution $\lambda$. Define for any $n\in\mathbb{N}$, the partial sum $S_n$ as
\begin{align*}
    S_n\coloneqq X_1+X_2+\cdots+X_n.
\end{align*}
Define, for any $t>0$, the stopping time $n_t$ as
\begin{align*}
    n_t\coloneqq \min{\{n\in\mathbb{N}:S_n\geq t\}}.
\end{align*}
Proposition~\ref{prop: refined Li-S Proposition 2.2} is essentially the same as \cite[Proposition 2.2]{li2022trigonometric}. 
\begin{prop}\label{prop: refined Li-S Proposition 2.2}
    Let $\lambda$ be a probability measure on $\mathbb{R}^+$ and $l\geq 2$. Suppose, $\lambda$ has finite support $\Lambda$, is non-lattice and $l$-weakly diophantine. Then, for any $C^1$ function $g:\mathbb{R}\to[0,+\infty)$ and $t>1+\max{\Lambda}$,
    \begin{align*}
        \mathbb{E}\left(g\left(S_{n_t}-t\right)\right)
        =\frac{1}{\sigma}\int_{\mathbb{R}^+}g(z)p(z)\,\mathrm{d}z+O\left(t^{-1/(4l+1)}\right)\|g\|_{C^1},
    \end{align*}
    where $\sigma$ is the expectation of $\lambda$, $p:\mathbb{R}^+\to[0,1]$ is the survival function of $\lambda$, that is, for any $z>0$, $p(z)\coloneqq\lambda((z,+\infty))$; and the local $C^1$-norm of $g$ is defined by
    \begin{align*}
        \|g\|_{C^1}
        \coloneqq
        \sup{\left\{|g(z)|+|g'(z)|:-1<z<1+\max{\Lambda}\right\}}
        .
    \end{align*}
\end{prop}

By combining Propositions~\ref{prop: refined Li-S Lemma 3.4} and \ref{prop: refined Li-S Proposition 2.2}, Proposition~\ref{prop: refined Li-S Proposition 3.5} is established to control the integral over points that are not near the diagonal. This proposition also incorporates optimisations to make the decay rate explicit, serving as a refined and improved version of \cite[Proposition 3.5]{li2022trigonometric}. These optimisations are achieved through carefully chosen parameters. 
\begin{prop}\label{prop: refined Li-S Proposition 3.5}
    Let $\alpha\in[0,1]$.
    Suppose, there exist $j,k\in\mathcal{A}$ and $l\geq 2$ such that $\log{r_j}/\log{r_k}$ is non-Liouville of degree $l$. Then, as $|\xi|\to\pm\infty$,
    \begin{align*}
        \iint_{[0,1]^2\setminus A_{t}}
        \sum_{w\in \mathcal{W}_t}p_we^{-2\pi i\xi (f_w(x)-f_w(y))}
        \,\mu(\mathrm{d}x)\,\mu(\mathrm{d}y)
        =O\left(\log^{-\alpha/((1+2\alpha)(8l-7))}{|\xi|}\right),
    \end{align*}
    where $A_{t}$ is defined in equation~\eqref{eq: def of A} and $t\coloneqq t(\xi)>1$ is the unique solution to the equation:
    \begin{align}
    \label{eq: C1}
        t^{(1+\alpha)/((1+2\alpha)(8l-7))}e^t=|\xi|.
    \end{align}
\end{prop}
\begin{proof}
    Pick any $\xi\in\mathbb{R}$. Suppose $|\xi|$ is large enough in terms of $\alpha$ and $l$, so that the unique solution $t\coloneqq t(\xi)>1$ to equation~\eqref{eq: C1} satisfies that $t>1+\max{\Lambda}$, where $\Lambda$ is the support of $\lambda$, defined in \eqref{eq: lambda}.
    
    Define a probability measure $\mathbb{P}_t$ on $\mathcal{A}^*\coloneqq\bigcup_{n\in\mathbb{N}}\mathcal{A}^n$ by
    \begin{align*}
        \mathbb{P}_t\coloneqq\sum_{w\in\mathcal{W}_t}p_w\delta_w,
    \end{align*}
    where $\delta$ is the Dirac delta measure. For any continuous function $g:\mathbb{R}\to\mathbb{C}$, the expectation
    \begin{align}
    \label{eq: 3.1 Li-S}
        \mathbb{E}\left(g\left(S_{n_t}-t\right)\right)
        =\int_{\mathcal{W}_t}g(-\log{r_w}-t)\,\mathbb{P}_t(\mathrm{d}w)
        =\sum_{w\in\mathcal{W}_t}p_wg(-\log{r_w}-t).
    \end{align}
    Define, for any $s\in\mathbb{R}$, a smooth function $g_s:\mathbb{R}\to\mathbb{C}$ by, for any $z\in\mathbb{R}$, 
    \begin{align*}
        g_s(z)\coloneqq\exp{\left(-2\pi ise^{-z}\right)}.
    \end{align*}
    In particular, as $s\to\pm\infty$, $\|g_s\|_{C^1}=O(|s|)$. By definition~\eqref{eq: f_w}, for any $x,y\in[0,1]$ and $w\in\mathcal{W}_t$, $e^{-2\pi i\xi (f_w(x)-f_w(y))}=e^{-2\pi i\xi(x-y)r_w}$.
    By equation~\eqref{eq: 3.1 Li-S} and $s\coloneqq\xi/e^t$, for any $x,y\in[0,1]$,
    \begin{align}
        \label{eq: 3.5 Li-S}
        \mathbb{E}\left(g_{(x-y)s}\left(S_{n_t}-t\right)\right)
        =
        \sum_{w\in \mathcal{W}_t}p_we^{-2\pi i\xi (f_w(x)-f_w(y))}
        .
    \end{align}

    By Proposition~\ref{prop: refined Li-S Lemma 3.4}, $\lambda$ is $(2l-2)$-weakly diophantine, defined as in \eqref{eq: lambda}. By Proposition~\ref{prop: refined Li-S Proposition 2.2}, 
    \begin{align*}
        \left|\mathbb{E}\left(g_{(x-y)s}\left(S_{n_t}-t\right)\right)-\frac{1}{\sigma}\int_{\mathbb{R}^+}g_{(x-y)s}(z)p(z)\,\mathrm{d}z\right|
        =O\left(\frac{|(x-y)s|}{t^{1/(4(2l-2)+1)}}\right)
        =O\left(\frac{|(x-y)s|}{t^{1/(8l-7)}}\right)
        .
    \end{align*}
    By substituting equation~\eqref{eq: 3.5 Li-S} into the above,
    \begin{align}
    \label{eq: renewal}
        \left|\sum_{w\in \mathcal{W}_t}p_we^{-2\pi i\xi (f_w(x)-f_w(y))}
        -\frac{1}{\sigma}\int_{\mathbb{R}^+}g_{(x-y)s}(z)p(z)\,\mathrm{d}z\right|
        =O\left(\frac{|(x-y)s|}{t^{1/(8l-7)}}\right).
    \end{align}
    
    Since $\Lambda$, the support of $\lambda$, is finite; the survival function $p:\mathbb{R}^+\to[0,1]$ of $\lambda$ is piecewise constant with finitely many points of discontinuity. By \cite[Lemma 3.8]{li2018decrease}, the decay rate in the main term, the oscillation integral, satisfies 
    \begin{align*}
        \left|\int_{\mathbb{R}^+}g_{(x-y)s}(z)p(z)\,\mathrm{d}z\right|=O\left(\frac{1}{|(x-y)s|}\right).
    \end{align*}
    By applying the triangle inequality, the above and equation~\eqref{eq: renewal} yield that
    \begin{align}
    \label{eq: two sums}
        \left|\sum_{w\in \mathcal{W}_t}p_we^{-2\pi i\xi (f_w(x)-f_w(y))}\right|
        =O\left(\frac{1}{|(x-y)s|}+\frac{|(x-y)s|}{t^{1/(8l-7)}}\right).
    \end{align}

    By the definition in equation~\eqref{eq: def of A}, for any $(x,y)\in{[0,1]^2\setminus A_{t}}$,
    \begin{align*}
        t^{\alpha/((1+2\alpha)(8l-7))}
        \leq |(x-y)s|
        \leq t^{(1+\alpha)/((1+2\alpha)(8l-7))},
    \end{align*}
    where $|s|=t^{(1+\alpha)/((1+2\alpha)(8l-7))}>1$.
    Since $\alpha\in[0,1]$, $l\geq2$, and $\xi=se^t$, 
    \begin{align*}
        \log{|\xi|}=\log{|s|}+t=\frac{1+\alpha}{(1+2\alpha)(8l-7)}\log{t}+t< 2t.
    \end{align*}
    The first term in the big-O notation in equation~\eqref{eq: two sums} is bounded above by
    \begin{align*}
        \frac{1}{|(x-y)s|}
        \leq t^{-\alpha/((1+2\alpha)(8l-7))}
        < 2^{1/27}\log^{-\alpha/((1+2\alpha)(8l-7))}{|\xi|}
        .
    \end{align*}
    Similarly, the second term is bounded above by
    \begin{align*}
        \frac{|(x-y)s|}{t^{1/(8l-7)}}
        \leq\frac{t^{(1+\alpha)/((1+2\alpha)(8l-7))}}{t^{1/(8l-7)}}
        \leq t^{-\alpha/((1+2\alpha)(8l-7))}
        < 2^{1/27}\log^{-\alpha/((1+2\alpha)(8l-7))}{|\xi|}
        .
    \end{align*}
    The result follows by the triangle inequality for integrals.
\end{proof}

It now remains to complete the proof of Theorem~\ref{thm: 1}.
\begin{proof}[Proof of Theorem~\ref{thm: 1}]
Let $\alpha\in[0,1]$ be an upper regularity exponent of $\mu$. Pick any $\xi\in\mathbb{R}$. Suppose $|\xi|$ is large enough in terms of $\alpha$ and $l$, so that the unique solution $t\coloneqq t(\xi)>1$ to equation~\eqref{eq: C1} satisfies both $t>1+\max{\Lambda}$ and $t>t_\alpha$, where $\Lambda$ is the support of $\lambda$, defined in \eqref{eq: lambda}, and $t_\alpha$ is given in Proposition~\ref{prop: refined Li-S Proposition 3.2}. Note that $\log{|\xi|}<2t$.

By applying Propositions~\ref{prop: refined Li-S Lemma 3.1}, \ref{prop: refined Li-S Proposition 3.2} and \ref{prop: refined Li-S Proposition 3.5},
\begin{align*}
    \left|\widehat{\mu}(\xi)\right|^2
    &\leq
    \left(\iint_{A_{t}}+\iint_{[0,1]^2\setminus A_{t}}\right)
    \sum_{w\in \mathcal{W}_t}p_we^{-2\pi i\xi (f_w(x)-f_w(y))}
    \, \mu(\mathrm{d}x) \, \mu(\mathrm{d}y) \\
    &=O\left( t^{-\alpha/((1+2\alpha)(8l-7))}\right)+O\left( \log^{-\alpha/((1+2\alpha)(8l-7))}{|\xi|} \right)
    =O\left(\log^{-\alpha/((1+2\alpha)(8l-7))}{|\xi|}\right).
\end{align*}
Theorem~\ref{thm: 1} is established.
\end{proof}

\section{Proof of Theorem~\ref{thm: 2}}
A few definitions and propositions are presented before starting the main proof of Theorem~\ref{thm: 2}. 
Let $\mathcal{A}$ be a non-empty finite set. Define $f_{\mathcal{A}}:[0,1]\to\mathbb{R}^+$ by, for any $s\in[0,1]$,
\begin{align}\label{eq: f_A}
    f_{\mathcal{A}}(s)
    \coloneqq\sum_{w\in\mathcal{A}}{r_w}^s.
\end{align}
Let $s\in[0,1]$. Define, for any $w\in\mathcal{A}$, a weight $p_{s,w}$ by
\begin{align}
\label{eq:p_s_w}
    p_{s,w}
    \coloneqq\frac{{r_w}^s}{f_{\mathcal{A}}(s)}
    .
\end{align}
Define a probability measure $\nu_{s}$ on $\mathcal{A}^\mathbb{N}$ by, for any $(a_n)_{n\in\mathbb{N}}\in\mathcal{A}^\mathbb{N}$ and $n\in\mathbb{N}$,
\begin{align*}
    \nu_s\left(\left\{(w_n)_{n\in\mathbb{N}}\in\mathcal{A}^\mathbb{N} : w_1 = a_1, w_2 = a_2, \ldots, w_n = a_n \right\} \right)
    \coloneqq p_{s,a_1} \, p_{s,a_2} \cdots p_{s,a_n}.
\end{align*}
Define a coding map $\pi:\mathcal{A}^\mathbb{N}\to F$ by, for any $w=(w_n)_{n\in\mathbb{N}}\in\mathcal{A}^\mathbb{N}$,
\begin{align*}
    \pi(w)
    \coloneqq \bigcap_{n\in\mathbb{N}}\left(f_{w_n} \circ \cdots \circ f_{w_2}\circ f_{w_1}\right)(F).
\end{align*}
By the assumption that the images $(f_w([0,1]))_{w\in\mathcal{A}}$ are pairwise disjoint except at the endpoints, the countable intersection in the expression of $\pi$ is a singleton in $F$.

Define a measure $\mu_{s}$ to be the push-forward measure of $\nu_s$ under the coding map $\pi$; that is, for any Borel subset $S\subset [0,1]$,
\begin{align}
\label{eq: mu definition}
    \mu_s(S)
    \coloneqq\left(\nu_s\circ\pi^{-1}\right)(S).
\end{align}
$\mu_s$ is a probability measure on $F$, as
\begin{align*}
    \mu_s(F)
    =\left(\nu_s\circ\pi^{-1}\right)(F)
    =\nu_s\left( \mathcal{A}^\mathbb{N} \right)
    =1
    ,
\end{align*}
and $\mu_s$ is self-similar, as
\begin{align*}
    \mu_s
    =\sum_{w\in\mathcal{A}}p_{s,w}\mu_{s}\circ {f_w}^{-1}.
\end{align*}

Proposition~\ref{prop: Holder} provides an upper regularity exponent for $\mu_s$. 
\begin{prop}\label{prop: Holder}
    Let $s\in[0,1]$. Suppose, $f_{\mathcal{A}}(s)\geq1$. Then, for any interval $I\subset [0,1]$, 
    \begin{align*}
        \mu_s(I)
        =O\left((\operatorname{diam}{I})^{s}\right),
    \end{align*}
    where $f_{\mathcal{A}}$ and $\mu_s$ are defined in \eqref{eq: f_A} and \eqref{eq: mu definition} respectively; and $\operatorname{diam}{I}$ denotes the diameter of $I$.
\end{prop}
\begin{proof}
    Pick any interval $I\subset[0,1]$. By the self-similarity property~\eqref{eq: SSS}, there exist a maximal $n\in\mathbb{N}\setminus\{0\}$ and $w=(w_1,w_2,\ldots,w_n)\in\mathcal{A}^{n}$ such that $I$ is covered by the $w$-cylinder; that is,
    \begin{align*}
        I\subset J_n\coloneq\left(f_{w_n}\circ\cdots\circ f_{w_2}\circ f_{w_1}\right)([0,1]).
    \end{align*}
    By convention, the composition of no functions is defined as the identity function. The interval $I$ intersects at most $\#\mathcal{A}$, the cardinality of $\mathcal{A}$, cylinders of the form
    \begin{align*}
        J_{n+1,w_{n+1}}\coloneq\left(f_{w_{n+1}}\circ f_{w_n}\circ\cdots\circ f_{w_2}\circ f_{w_1}\right)([0,1]),
    \end{align*}
    where $w_{n+1}\in\mathcal{A}$.
    
    Suppose $I$ intersects exactly one $J_{n+1}\coloneqq J_{n+1,w_{n+1}}$ for some $w_{n+1}\in\mathcal{A}$. Without loss of generality, by discarding sub-intervals of $I$ that contribute zero $\mu_s$-measure to $I$, $I$ satisfies $J_{n+1}\subset I\subset J_n$. Thus, $\operatorname{diam}J_{n+1}\leq h\coloneqq \operatorname{diam}I\leq \operatorname{diam}J_n=r_{w_1}r_{w_2}\cdots r_{w_n}$. By the assumption that $f_{\mathcal{A}}(s)\geq1$ and the definition in equation~\eqref{eq:p_s_w}, for any $w\in\mathcal{A}$, $p_{s,w}\leq {r_w}^s$. Thus, 
    \begin{align*}
        \mu_{s}(I)
        &
        \leq \mu_{s}(J_n)
        \leq p_{s,w_1}p_{s,w_2}\cdots p_{s,w_n}
        \leq {r_{w_1}}^{s}{r_{w_2}}^{s}\cdots{r_{w_n}}^{s} \\
        &
        \leq {r_{w_{n+1}}}^{-s}\left(r_{w_1}r_{w_2}\cdots r_{w_{n+1}}\right)^{s}
        \leq {r_{w_{n+1}}}^{-s}h^{s}
        \leq D_sh^{s},
    \end{align*}
    where $D_s\coloneqq\max_{w\in\mathcal{A}}{{r_w}^{-s}}$.
    
    Suppose $I$ intersects more than one $J_{n+1,w_{n+1}}$.
    By the sub-additivity of measures, 
    \begin{align*}
        \mu_{s}(I)
        \leq \sum_{w_{n+1}\in\mathcal{A}}\mu_{s}(J_{n+1,w_{n+1}})
        \leq C_sh^{s},
    \end{align*}
    where $C_s\coloneqq D_s\cdot\#\mathcal{A}$.
\end{proof}

\begin{prop}\label{prop: 7}
    There exists a unique $s=s_\mathcal{A}\in[0,1]$ such that equation~\eqref{eq: sum rw s =1} is satisfied.
\end{prop}
\begin{proof}
    The existence follows from the Intermediate Value Theorem. Since $f_{\mathcal{A}}$ is a finite sum of continuous functions on $[0,1]$, it is itself continuous on $[0,1]$. Notice that, $f_{\mathcal{A}}(0)=\#\mathcal{A}\geq1$ and $f_{\mathcal{A}}(1)=\sum_{w\in\mathcal{A}}r_w\leq1$. Otherwise, the total length of the image intervals $(f_w([0,1]))_{w\in\mathcal{A}}$ would exceed 1, resulting in an uncountable intersection, which contradicts the assumption. By the Intermediate Value Theorem, there exists $s_\mathcal{A}\in[0,1]$ such that $f_{\mathcal{A}}(s_\mathcal{A})=1$. 

    The uniqueness follows from the strict monotonicity of $f_{\mathcal{A}}$. For any $w\in\mathcal{A}$, $r_w\in(0,1)$ and consequently $\log{r_w}<0$. Thus, for any $s\in[0,1]$, the first derivative of $f_{\mathcal{A}}$,
    \begin{align*}
        {f_{\mathcal{A}}}'(s)
            =\sum_{w\in\mathcal{A}}{r_w}^s\log{r_w}<0.
    \end{align*}
    It follows that $f_{\mathcal{A}}$ is strictly decreasing on $[0,1]$. Hence, the equation $f_{\mathcal{A}}(s)=1$ has at most one solution for $s\in[0,1]$.
\end{proof}

It now remains to complete the proof of Theorem~\ref{thm: 2}; this is the unique solution to equation~\eqref{eq: sum rw s =1} given by Proposition~\ref{prop: 7} is the Hausdorff dimension. The proof is separated into two parts, establishing the inequalities $\dim{F}\leq s_\mathcal{A}$ and $\dim{F}\geq s_\mathcal{A}$. The first inequality is proved by a standard covering argument, and the second is proved by applying Proposition~\ref{prop: Holder}.

\begin{proof}[Proof of Theorem~\ref{thm: 2}]
Define $r^*\coloneqq\max_{w\in\mathcal{A}}{r_w}\in(0,1)$. Pick any $s>s_\mathcal{A}$, $\varepsilon>0$, and $\rho>0$. For these parameters, take
\begin{align*}
    k\coloneqq\max{\left\{1,\left\lceil\frac{\log{\rho}}{\log{r^*}}\right\rceil,\left\lceil\frac{\log{\varepsilon}}{\log{f_{\mathcal{A}}(s)}}\right\rceil\right\}}\in\mathbb{N},
\end{align*}
where $f_{\mathcal{A}}$ is defined in equation~\eqref{eq: f_A}. Define $\mathcal{F}_k$ to be the set of all cylinders at level $k$ by
\begin{align*}
    \mathcal{F}_k
    \coloneqq
    \left\{\left(f_{w_k}\circ\cdots \circ f_{w_2}\circ  f_{w_1}\right)([0,1])\subset[0,1]:(w_1,w_2,\ldots,w_k)\in\mathcal{A}^k
    \right\}
    ,
\end{align*}
where each $f_w$ is given by equation~\eqref{eq: IPS}. By the self-similarity property~\eqref{eq: SSS}, $\mathcal{F}_k$ is a cover of $F$. There are ${(\#\mathcal{A})}^k$ pairwise disjoint intervals (except possibly at the endpoints) in $\mathcal{F}_k$, each of length at most $\rho$, as for any $(w_1,w_2,\ldots,w_k)\in\mathcal{A}^k$,
\begin{align*}
    \operatorname{diam}{\left(f_{w_k}\circ\cdots \circ f_{w_2}\circ  f_{w_1}\right)([0,1])}
    =r_{w_1}r_{w_2}\cdots r_{w_k} 
    \leq (r^*)^k 
    \leq \rho.
\end{align*}
Then, $\mathcal{F}_k$ is a $\rho$-cover of $F$, and the $s$-dimensional Hausdorff measure of $F$ is at most
\begin{align*}
    \sum_{I\in \mathcal{F}_k}(\operatorname{diam}{I})^s
    &=\sum_{w_1\in\mathcal{A}}
    \sum_{w_2\in\mathcal{A}}\cdots\sum_{w_k\in\mathcal{A}}\left(r_{w_1}r_{w_2}\cdots r_{w_k}\right)^s 
    =\left(\sum_{w\in\mathcal{A}}{r_w}^s\right)^k 
    = \left(f_{\mathcal{A}}(s)\right)^k 
    \leq\varepsilon,
\end{align*} 
as $f_{\mathcal{A}}$ is strictly decreasing on $[0,1]$, and $s>s_\mathcal{A}$ implies that $f_{\mathcal{A}}(s)<f_{\mathcal{A}}(s_\mathcal{A})=1$. Therefore, the $s$-Hausdorff measure of $F$ is 0, and consequently $\dim{F}\leq s_\mathcal{A}$.

Pick any $s<s_\mathcal{A}$ and countable collection of intervals $(I_j)_{j\in\mathbb{N}}$. Suppose $(I_j)_{j\in\mathbb{N}}$ covers $F$, that is $F\subset\bigcup_{j\in\mathbb{N}}I_j$. Since $f_{\mathcal{A}}$ is strictly decreasing on $[0,1]$, $f_{\mathcal{A}}(s)>f_{\mathcal{A}}(s_\mathcal{A})=1$. By Proposition~\ref{prop: Holder},
\begin{align*}
    \sum_{j\in\mathbb{N}}(\operatorname{diam}{I_j})^s
    \gg \sum_{j\in\mathbb{N}}\mu_s{(I_j)}
    \geq\mu_s(F)=1.
\end{align*}
Therefore, the $s$-Hausdorff measure of $F$ is positive, and consequently $\dim{F}\geq s_\mathcal{A}$.

Theorem~\ref{thm: 2} is established.
\end{proof}

\section{Proof of Theorem~\ref{thm: 3}}
\begin{proof}[Proof of Theorem~\ref{thm: 3}]
By Theorem~\ref{thm: 2} and the definition in equation~\eqref{eq: f_A}, $f_{\mathcal{A}}(\dim{F})=1$. By Proposition~\ref{prop: Holder}, for any interval $I\subset[0,1]$,
\begin{align*}
    \mu_{\mathcal{A}}(I)
    = O\left((\operatorname{diam}{I})^{\dim{F}}\right),
\end{align*}
where $\mu_{\mathcal{A}}\coloneqq\mu_{\dim{F}}$, as defined in equation~\eqref{eq: mu definition}. Thus, the Hausdorff dimension of $F$ is the maximal upper regularity exponent of $\mu_{\mathcal{A}}$. By setting $\alpha\coloneqq \dim{F}$ in Theorem~\ref{thm: 1}, the desired decay rate in Theorem~\ref{thm: 3} is established. 
\end{proof}

\section{Proof of Theorem~\ref{thm: 4}}
By equations~\eqref{eq: Luroth SSS} and \eqref{eq: Luroth IFS}, it is evident that for any finite $\mathcal{A}\subset\mathbb{N}\setminus\{1\}$, the set $L_{\mathcal{A}}\subset[0,1]$ is self-similar, and the images $(f_w([0,1]))_{w \in \mathcal{A}}$ are pairwise disjoint except possibly at the endpoints. To establish Theorem~\ref{thm: 4}, as an application of Theorem~\ref{thm: 3}, it suffices to prove the remaining non-Liouville condition for non-singleton $\mathcal{A}\subset\mathbb{N}\setminus\{1\}$.

First, the linear independence of logarithms of the contraction ratios is proved. Then, the explicit Baker Theorem by Matveev \cite[Theorem 2.1]{matveev2000explicit} is applied, establishing the non-Liouville degree explicitly.

\begin{prop}\label{prop: Luroth linearly independent}
    For any $a_1,a_2\in\mathbb{N}\setminus\{1\}$, if $a_1\ne a_2$ then $\log{(a_1(a_1-1))}$ and $\log{(a_2(a_2-1))}$ are linearly independent over $\mathbb{Z}$.
\end{prop}
\begin{proof}
    The contrapositive of the statement is proved. Suppose there exists $k\in\mathbb{Q}$ such that 
    \begin{align}
    \label{eq: Luroth exp form}
        a_1(a_1-1)=(a_2(a_2-1))^k\geq2.
    \end{align}

    Suppose there exist $n,m\in\mathbb{N}$ such that $n\geq2$ and $a_1(a_1-1)=n^m$. Let ${p_1}^{e_1}{p_2}^{e_2}\cdots {p_s}^{e_s}$ be the prime factorisation of $n$, where for any $i\in\mathbb\{1,2,\ldots,s\}$, $p_i$ is a distinct prime and $e_i\in\mathbb{N}$.
    Since $\gcd(a_1,a_1-1)=1$, without loss of generality (if necessary, relabel the primes accordingly), there exists $r\in\{1,2,\ldots,s\}$ such that $a_1=({p_1}^{e_1}{p_2}^{e_2}\cdots {p_r}^{e_r})^m$ and $a_1-1=({p_{r+1}}^{e_{r+1}}{p_{r+2}}^{e_{r+2}}\cdots {p_s}^{e_s})^m$. Thus, $a_1$ and $a_1-1$ are consecutive positive integers, and both are $m$-th powers of positive integers. The only possibility is $m=1$.

    Therefore, both $a_1(a_1-1)$ and $a_2(a_2-1)$ are not perfect powers. By equation~\eqref{eq: Luroth exp form}, the only possibility is $k=1$. Hence, $a_1(a_1-1)=a_2(a_2-1)$, and consequently, $a_1=a_2$.
\end{proof}
\begin{prop}\label{prop: Luroth SSL}
For any $a_1,a_2\in\mathbb{N}\setminus\{1\}$, if $a_1\ne a_2$ then
\begin{align*}
    \frac{\log{(a_1(a_1-1))}}{\log{(a_2(a_2-1))}}
\end{align*}
is non-Liouville of degree $l_{a_1,a_2}$, where
\begin{align}
    \label{eq: l a1 a2}
    l_{a_1,a_2}
    &\coloneqq387072e^3(15.8+5.5\log{2})\log{(a_1(a_1-1))}\log{(a_2(a_2-1))}+1 \\
    &>189369098. \nonumber
\end{align}
\end{prop}
\begin{proof}
Pick any $p\in\mathbb{Z}$ and $q\in\mathbb{N}$. By Proposition~\ref{prop: Luroth linearly independent} and \cite[Theorem 2.1]{matveev2000explicit},
\begin{align}
\label{eq: Luroth log linear}
    L(p,q,a_1,a_2)
    \coloneqq\log{|q\log{(a_1(a_1-1))}-p\log{(a_2(a_2-1))}|} >-C(n)C_0W_0D^2\Omega,
\end{align}
where the values in \cite[Theorem 2.1]{matveev2000explicit} are $n=2$, $D=1$, $x=1$, and
\begin{align*}
    C(n) &= 387072e^3; \\
    C_0  &= 15.8+5.5\log{2}; \\
    B &=\max\left\{1,|p|,\frac{\log{(a_1(a_1-1))}}{\log{(a_2(a_2-1))}}q\right\}\geq\frac{\log{(a_1(a_1-1))}}{\log{(a_2(a_2-1))}}q; \\
    W_0  &= \log{\left(\frac{3e\log{(a_1(a_1-1))}}{2\log{(a_2(a_2-1))}}q\right)}; \\
    \Omega &=\log{(a_1(a_1-1))}\log{(a_2(a_2-1))}.
\end{align*}
By taking the exponential function on both sides of equation~\eqref{eq: Luroth log linear},
\begin{align*}
    \exp{L(p,q,a_1,a_2)}
    &=\left|q\log{(a_1(a_1-1))}-p\log{(a_2(a_2-1))}\right| \\
    &>\left(\frac{3e\log{(a_1(a_1-1))}}{2\log{(a_2(a_2-1))}}q\right)^{-C(n)C_0\Omega} \\
    &> c_{a_1,a_2}q^{-387072e^3(15.8+5.5\log{2})\log{(a_1(a_1-1))}\log{(a_2(a_2-1))}}\log{(a_2(a_2-1))},
\end{align*}
where
\begin{align*}
    c_{a_1,a_2}
    \coloneqq\frac{1}{\log{(a_2(a_2-1))}}\left(\frac{3e\log{(a_1(a_1-1))}}{2\log{(a_2(a_2-1))}}\right)^{-C(n)C_0\Omega}>0.
\end{align*}
By dividing both sides by $q\log{(a_2(a_2-1))}$,
\begin{align*}
    \left|\frac{\log{(a_1(a_1-1))}}{\log{(a_2(a_2-1))}}-\frac{p}{q}\right|&>\frac{c_{a_1,a_2}}{q^{l_{a_1,a_2}}},
\end{align*}
where $l_{a_1,a_2}$ is defined in equation~\eqref{eq: l a1 a2}.
\end{proof}

Proposition~\ref{prop: 10} offers a faster decay rate than Theorem~\ref{thm: 4}, but its decay parameter involves a more complicated expression. Theorem~\ref{thm: 4} is derived from Proposition~\ref{prop: 10}, provides a more accessible formula, and facilitates understanding how the decay rate varies with the parameters.

\begin{prop}\label{prop: 10}
    For any non-singleton finite $\mathcal{A}\subset\mathbb{N}\setminus\{1\}$, there exists a self-similar probability measure $\mu_{\mathcal{A}}$ on $L_{\mathcal{A}}$ such that as $\xi\to\pm\infty$,
    \begin{align*}
        \widehat{\mu_{\mathcal{A}}}(\xi)=O\left(\log^{-\beta_{\mathcal{A},0}}{|\xi|}\right),
    \end{align*}
    where $a_1$ and $a_2$ are the first smallest and second smallest elements in $\mathcal{A}$ respectively, and
    \begin{align*}
        \beta_{\mathcal{A},0}
        \coloneqq
        \frac{1}{2}\frac{\dim{L_\mathcal{A}}}{1+2\dim{L_\mathcal{A}}}\frac{1}{3096576e^3(15.8+5.5\log{2})\log{(a_1(a_1-1))}\log{(a_2(a_2-1))}+1}
        >0
        .
    \end{align*}
\end{prop}
\begin{proof}
By the formula for the non-Liouville degree in equation~\eqref{eq: l a1 a2}, Theorem~\ref{thm: 3} gives the desired decay rate after substitution and calculation. The fastest decay rate under Theorem~\ref{thm: 3} is achieved when $l_{a_1,a_2}$ is minimised. Since $l_{a_1,a_2}$ is jointly and strictly increasing, $a_1$ and $a_2$ are selected to be the first smallest and the second smallest elements in $\mathcal{A}$ respectively.
\end{proof}

\begin{proof}[Proof of Theorem~\ref{thm: 4}]
Theorem~\ref{thm: 4} follows directly from Proposition~\ref{prop: 10} and that $\beta_{\mathcal{A},0}\geq\beta_{\mathcal{A}}$.
\end{proof}

\bibliographystyle{siam}
\bibliography{name}

\begin{thebibliography}{10}

\bibitem{beresnevich2016metricdiophantineapproximationaspects}
{\sc V.~Beresnevich, F.~Ramírez, and S.~Velani}, {\em Metric Diophantine Approximation: Aspects of Recent Work}, London Mathematical Society Lecture Note Series, Cambridge University Press, 2016, p.~1–95.

\bibitem{FENG2009407}
{\sc D.-J. Feng and K.-S. Lau}, {\em Multifractal formalism for self-similar measures with weak separation condition}, Journal de Math\'matiques Pures et Appliqu\'es, 92 (2009), pp.~407--428.

\bibitem{HENSLEY1989182}
{\sc D.~Hensley}, {\em The {Hausdorff} dimensions of some continued fraction {Cantor} sets}, Journal of Number Theory, 33 (1989), pp.~182--198.

\bibitem{hensley1996polynomial}
\leavevmode\vrule height 2pt depth -1.6pt width 23pt, {\em A polynomial time algorithm for the {Hausdorff} dimension of continued fraction {Cantor} sets}, Journal of Number Theory, 58 (1996), pp.~9--45.

\bibitem{hutchinson1981fractals}
{\sc J.~E. Hutchinson}, {\em Fractals and self similarity}, Indiana University Mathematics Journal, 30 (1981), pp.~713--747.

\bibitem{kaufman1980continued}
{\sc R.~Kaufman}, {\em Continued fractions and {Fourier} transforms}, Mathematika, 27 (1980), pp.~262--267.

\bibitem{li2018decrease}
{\sc J.~Li}, {\em Decrease of {Fourier} coefficients of stationary measures}, Mathematische Annalen, 372 (2018), pp.~1189--1238.

\bibitem{li2022trigonometric}
{\sc J.~Li and T.~Sahlsten}, {\em Trigonometric series and self-similar sets}, Journal of the European Mathematical Society (EMS Publishing), 24 (2022).

\bibitem{LyonsRussell1986Tmon}
{\sc R.~Lyons}, {\em The measure of non-normal sets}, Inventiones mathematicae, 83 (1986), pp.~605--616.

\bibitem{MattilaPertti1995Gosa}
{\sc P.~Mattila}, {\em Geometry of Sets and Measures in {Euclidean} Spaces: Fractals and Rectifiability}, Cambridge: Cambridge University Press, Cambridge [England]; New York, 1995.

\bibitem{matveev2000explicit}
{\sc E.~M. Matveev}, {\em An explicit lower bound for a homogeneous rational linear form in the logarithms of algebraic numbers. ii}, Izvestiya: Mathematics, 64 (2000), pp.~1217--1269.

\bibitem{Moran_1946}
{\sc P.~A.~P. Moran}, {\em Additive functions of intervals and {Hausdorff} measure}, Mathematical Proceedings of the Cambridge Philosophical Society, 42 (1946), pp.~15--23.

\bibitem{pollington2020inhomogeneous}
{\sc A.~Pollington, S.~Velani, A.~Zafeiropoulos, and E.~Zorin}, {\em Inhomogeneous {Diophantine} approximation on {$M_0$}-sets with restricted denominators}, International Mathematics Research Notices, 2022 (2020), pp.~1--73.

\bibitem{queffelec2003analyse}
{\sc M.~Queff{\'e}lec and O.~Ramar{\'e}}, {\em Analyse de {Fourier} des fractions continues {\`a} quotients restreints}, Enseignement Math\'ematique, 49 (2003), pp.~335--356.

\end{thebibliography}

\end{document}